%% file: del2.tex
\documentstyle[jalc]{article}
\input notation.tex

\newtheorem{The}{Theorem}
\newtheorem{Lem}[The]{Lemma}

\theorembodyfont{\rm}

\newtheorem{Rem}{Remark}

\title{A NEW MEASURE OF ASYMMETRY OF BINARY WORDS}
\author{Olexandr Ravsky}
\address{Department of Mathematics, Ivan Franko Lviv National University \\
      Universytetska 1, Lviv, Ukraine}
\email{oravsky@mail.ru}

\abstract{A binary word is symmetric if it is a palindrome or an
antipalindrome. We define a new measure of asymmetry of a binary word equal to
the minimal number of letters of the word whose deleting from the word yields a
symmetric word and obtain upper and lower estimations of this measure.}
\keywords{palindrome, measure of symmetry, measure of asymmetry.}

\begin{document}
\maketitle

In this paper under the term word we shall understood the binary word over the
alphabet $\{a,b\}$. The set of all words can naturally be thought as a free
semigroup with two generators, which we shall use in the notation of the words.
There are various approaches to a measure of asymmetry of a word or a group
colored into two colors \cite{2}, \cite{4} (see \cite{1} for more references).
In the present paper we investigate a new measure of symmetry of a word,
suggested by the question of Ihor Protasov \cite{3}. By a {\it symmetric word}
we shall understood a word which is a palindrome or an antipalindrome. A word
$w=a_1\cdots a_n$ is called a {\it palindrome} (respectively, an {\it
antipalindrome}) if $a_i=a_{n-i+1}$ (respectively, $a_i\not=a_{n-i+1}$) for all
$i\le n.$ Given a word $w$ let $S_d(w)$ be the minimal number of letters of $w$
whose deleting from $w$ yields a palindrome or an antipalindrome.  Observe that
a word $w$ is symmetric if and only if $S_d(w)=0$. Thus the number $S_d(w)$ can
be thought as an asymmetry measure of $w$. For every positive integer $n$ let
$S_d(n)=\max\{S_d(w):w\in S$, $l(w)=n\}$ be the maximal asymetry measure
$S_d(w)$ of a word $w$ of length $l(w)=n$.

Ihor Protasov observed that $S_d(n)\le n/3$ for small $n$ and asked in \cite{3}
if this estimation holds for every $n$. Computer calculations show that this
conjecture fails already for $n=10.$ The values of $S_d(n)$ for $n\le 20$ are
in the table on the next page.
\begin{table}[ht]
\begin{center}
\begin{tabular}{|c|c|c|c|c|c|c|c|c|c|c|c|c|c|c|c|c|c|c|c|c|}
\hline
$n$ &1 & 2 & 3 & 4 & 5 & 6 & 7 & 8 & 9 & 10\\
\hline
$S_d(n)$ & 0 & 0 &1 & 1 & 1 & 2 & 2 & 2 & 3 & 4\\
\hline
 & 11 & 12 & 13 & 14 & 15 & 16 & 17 & 18 & 19 & 20\\
\hline
 & 4 & 4 & 5 & 5 & 5 & 6 & 7 & 7 & 7 & 8\\
\hline
\end{tabular}
\end{center}
\end{table}

The main result of the paper is

\begin{The} For every $n\ge 2$ the number $S_d(n)$ lies in the range
\[
\left[\frac{n+2\left[\frac{n-3}{7}\right]}{3}\right]\le S_d(n)\le \left[\frac {n}{2}\right].
\]
\end{The}
In order to prove the main theorem we need some lemmas.

\begin{Lem} Let $w=a_1\cdots a_n$ be a word with $a_1=a_n$.
If after deleting some $k$ letters we obtain a palindrome, then we can obtain a
palindrome by deleting $l\le k$ letters such that the letters $a_1,a_n$ rest
undeleted.
\end{Lem}
\begin{proof} If after deleting some $k$ letters we obtain a palindrome,
and both of the letters $a_1$ and $a_n$ were deleted, then we simply undelete
these letters and obtain a longer palindrome. If only one of the letters $a_1$
and $a_n$ (say, $a_n$) was deleted then there exists $m<n$ such that the
letters $a_{m+1}\dots a_n$ were deleted and $a_m$ was not. Then $a_m=a_1$ and
we may delete the letter $a_m$ instead of the letter $a_n.$
\end{proof}

\begin{Lem} Let  $w=a_1\cdots a_n$ be a word with $a_1\not=a_n.$
If after deleting some $k$ letters we obtain an antipalindrome, then we can
obtain an antipalindrome by deleting $l\le k$ letters such that the letters
$a_1,a_n$ rest undeleted.
\end{Lem}
\begin{proof} It is similar to that of Lemma 2.
\end{proof}

The first $[l(w)/2]$ letters of a word $w$ are called {\it the first half} and
the last $[l(w)/2]$ letters of $w$ {\it the second half} of the word $w.$

\begin{Lem} Let $n\in\N$ and $(\alpha,\beta)\in\{(0,0),(1,0),(1,1),(2,1),(3,1),(3,2),(4,2)\}.$
Let $w_{n,\alpha,\beta}=b^{n+1}(ab)^nb^{2n+1+\alpha}a^{2n+1+\beta}.$
Then $S_d(w_{n,\alpha,\beta})\ge 3n+1+[(\alpha+\beta)/3].$
\end{Lem}
\begin{proof}
Note that
$l(w_{n,\alpha,\beta})=n+1+2n+2n+1+\alpha+2n+1+\beta=7n+3+\alpha+\beta.$ It is
easy to check that for all $\alpha,\beta$ from the lemma's condition we have
$[(\alpha+\beta)/3]\le \beta$ and $[(\alpha+\beta)/3]\le\alpha-\beta.$ Suppose
that we delete $k$ letters from the word $w_{n,\alpha,\beta}$ and obtain a
final palindrome $p.$ We shall show that $k\ge 3n+1+[(\alpha+\beta)/3].$ Since
a palindrome has the same first and last letters, either the first $n+1$
letters $b,$ or the last $2n+1+\beta$ letters $a$ are deleted.

Assume that the first $n+1$ letters $b$ were deleted. Suppose that $l(p)>
l(w_{n,\alpha,\beta})-(3n+1)-[(\alpha+\beta)/3]=
4n+2+\alpha+\beta-[(\alpha+\beta)/3]\ge 4n+2+\alpha+\beta-\beta>2(2n).$ Then
every letter from the subword $(ab)^n$ either is deleted, or belongs to the
first half of the final palindrome. Therefore the second half of the final
palindrome is equal to $b^{k_b}a^{k_a}$ and the first half of the final
palindrome is equal to $a^{k_a}b^{k_b} $ for some nonnegative numbers $k_a$ and
$k_b$. If in the subword $(ab)^n$ it was deleted $i_a$ letters $a$ and $i_b$
letters $b$ then $i_a+i_b\ge n-1.$ If $k_b=0$ then $l(p)\le
n+2n+1+\beta+1<4n+2+\alpha,$ a contradiction. Therefore there exists a letter
$b$ at the second half and only $n-i_a$ letters from the subword
$a^{2n+1+\beta}$ have a pair in the first half of the final palindrome. Thus
other $2n+1+\beta-(n-i_a)$ letters $a$ of the subword $a^{2n+1+\beta}$ are
deleted. Thus in this case $l(p)\le 7n+3+\alpha+\beta-((n+1)+(i_a+i_b)+
(2n+1+\beta-(n-i_a)))=5n+1+\alpha-2i_a-i_b\le 4n+2+\alpha-i_a \le
4n+2+\alpha+\beta-[(\alpha+\beta)/3],$ a contradiction. Therefore $l(p)\le
l(w_{n,\alpha,\beta})-(3n+1)-[(\alpha+\beta)/3]$ and $k\ge
3n+1+[(\alpha+\beta)/3].$

Assume now that the last $2n+1+\beta$ letters $a$ are deleted. Then Lemma 2
implies that it suffices to show that we cannot delete
$l_d<n+[(\alpha+\beta)/3]-\beta$ letters from the word $(ab)^nb^{n+\alpha}$ and
obtain a palindrome $q$. If $l(q)$ is an odd number and we have $m$ letters $a$
in the second half then there is at most $n-m-1$ letters $b$ in the first half
of $q$ and therefore $l(q)\le 1+2m+2(n-m-1)=2n-1.$ If $l(q)$ is even and we
have $m$ letters $a$ in the second half of $q,$ then there are at most $n-m$
letters $b$ in the first half and therefore $l(q)\le 2m+2(n-m)=2n.$ In the both
cases $l_d\ge 3n+\alpha-l(q)\ge n+\alpha\ge n+[(\alpha+\beta)/3]-\beta.$

Suppose that we delete $k$ letters from the word $w_{n,\alpha,\beta}$ and obtain an
antipalindrome. We shall show that $k\ge 3n+1+[(\alpha+\beta)/3].$ Lemma 3
implies that it suffices to show that we cannot delete
$l<3n+1+[(\alpha+\beta)/3]$ letters from the word
$(ab)^nb^{2n+1+\alpha}a^{n+\beta}$ and obtain an antipalindrome $p$.

Suppose that all the letters of the second half of the final antipalindrome $p$
belongs to the word $b^{2n+1+\alpha}a^{n+\beta}.$ Therefore the first half of
$p$ has a representation $b^{k_a}a^{k_b} $ similar to representation
$b^{k_b}a^{k_a}$ of its second half. If $k_b=0$ then $l(p)\le 2(n+\beta)$ and
$l\ge 5n+1+\alpha+\beta-2n-2\beta= 3n+1+\alpha-\beta\ge
3n+1+[(\alpha+\beta)/3].$ If $k_b>0$ then in the subword $(ab)^n$ it was
deleted $i_a\ge k_a$ letters $a.$ Since $k_b\le n-i_a,$ we obtain that $l(p)\le
2(k_a+k_b)\le 2(k_a+n-k_a)=2n.$ Thus in this case $l\ge 5n+1+\alpha+\beta-2n\ge
3n+1+[(\alpha+\beta)/3].$

Suppose that in the second half of the final antipalindrome $p$ there are $m>0$
letters of the word $(ab)^n.$ Therefore the word consisting of the first
$l(p)/2-m$ letters of $p$ is equal to $b^{k_a}a^{k_b},$ while the word
consisting of the last $l(p)/2-m$ letters is equal to $b^{k_b}a^{k_a}.$ Then in
the subword $(ab)^n$ were deleted $i_a\ge k_a$ letters $a$ and $i_b\ge k_b-1$
letters $b.$ Consequently, $l(p)\le 2n-i_a-i_b+k_a+k_b\le 2n+1.$ Since $l(p)$
is even $l(p)\le 2n.$ Thus in this case $l\ge 5n+1+\alpha+\beta-2n\ge
3n+1+[(\alpha+\beta)/3].$
\end{proof}

Now we can prove Theorem 1.
\begin{proof} To prove that $S_d(n)\le \left[\frac n2\right]$ it
suffices to remark that we  can always delete all letters $a$ or all letters
$b$ and obtain a palindrome. Let $S_d'(n)=\left[\frac{n+2[(n-3)/7]}{3}\right].$
Lemma 4 yields that for every numbers $t\ge 1,0\le k\le 6$ we have
$S_d(7t+3+k)\ge 3t+1+[k/3].$ The computer calculation shows that it is true
also for $t=0.$ On the other hand, $S_d'(7t+3+k)=[(7t+3+k+2t)/3]=3t+1+[k/3].$
\end{proof}

The referee noted the following 

\begin{Rem} It is easy to check that the inequality in the lemma is in fact 
an equality. This is of course not needed to prove the theorem, 
but reveals that there is no hope to get a better bound using 
the same words. 
\end{Rem}

\begin{Rem} For all $2\le n\le 20$ the lower bound from Theorem 1 is exact.
\end{Rem}

\begin{Rem}
Like in many other related with symmetry problems, we may consider the
following antagonistic game. Let $n\in\N.$ The first player selects a word of
the length $n.$ Then beginning with the second player, two players consequently
delete letters of this word. The game ends when the word becomes a palindrome
or an antipalindrome. The gain $g_1$ of the first player is the number of
moves.

 Consider the following strategy of the first player for $n\ge 6.$
An the beginning the first player selects the word $a^kb^{k+2}$ for even $n$
and the word $a^kb^{k+3}$ for odd $n.$ If the second player has deleted a
letter $a$ then the first player deletes a letter $b$ and conversely. This
strategy yields an estimation $g_1(n)\ge n-4\gg S_d(n).$
\end{Rem}

\end{document}

%% file: notation.tex
\def\phi{\varphi}
\def\0{\varnothing}

\def\N{\bf N}